\documentclass{rspublic}
\usepackage{graphicx,natbib}

\newcommand{\x} {{\bf x}}

\newcommand{\apriori}{{\it a priori}~}
\newcommand{\ie}{{\it i.e.}~}
\newcommand{\eg}{{\it e.g.}~}

\newcommand{\gest}{\bar{g}^{\text{est}}}

\newcommand{\ave}[1]{\left\langle#1 \right\rangle}
\newcommand{\elabel}[1]{\label{eq:#1}}
\newcommand{\eref}[1]{(Eq.~\ref{eq:#1})}

\newcommand{\Eref}[1]{Equation~(\ref{eq:#1})}

\newcommand{\flabel}[1]{\label{fig:#1}}
\newcommand{\fref}[1]{Fig.~\ref{fig:#1}}

\begin{document}
\vsize21.5cm 

%  \title[Short title]{Full title}
%short title < 40 char =running head
%   \author[Short author names]{Full author names}
%if > 40 char 1st author + 'and others'
%   \affiliation{Full postal addresses of all authors}
%There may also be a present address (as footnote), if different to the place where the work was done.
 
\title[Time resolution of St. Petersburg paradox]{The time resolution of the St. Petersburg paradox}

\author[O. Peters]{Ole Peters$^{1,2,3}$}

\affiliation{ $^{1}$Dept. of Mathematics and Grantham Institute for Climate Change, Imperial College London, SW7
  2AZ, UK.\\
and\\
$^{2}$Dept. of Atmospheric and Oceanic Sciences,
University of California Los Angeles,
7127 Math. Sci. Bldg.,
405 Hilgard Ave.,
Los Angeles, CA, 90095-1565, USA,\\
and\\
$^{3}$Santa Fe Institute,
1399 Hyde Park Rd., 
Santa Fe, NM, 87501, 
USA}

\label{firstpage}

\maketitle

\begin{abstract}%
{St. Petersburg paradox, expectation, ergodicity, risk,
  utility}%keywords 

A resolution of the St. Petersburg paradox is presented. In contrast
to the standard resolution, utility is not required. Instead, the
time-average performance of the lottery is computed. The final result
can be phrased mathematically identically to Daniel Bernoulli's
resolution, which uses logarithmic utility, but is derived using a
conceptually different argument. The advantage of the time resolution is
the elimination of arbitrary utility functions.
\end{abstract}

\bigskip
\noindent{\small Submitted to  \textit{Philosophical Transactions of the Royal Society A.} \\ Theme Issue `Handling Uncertainty in Science'.} 
\vspace{1cm}

I investigate practical consequences of a radical idea built into the
foundations of probability theory. The idea is that of embedding a
stochastic system in an ensemble of systems which all start in the
same state but develop along different trajectories. To understand how
this idea was absorbed into the theory, the original motivation for
developing the concept of probability and expectation values is
reviewed in Sec.~\ref{Origins}. Section~\ref{lottery} describes the
St. Petersburg paradox, the first well documented example of a
situation where the use of ensembles leads to absurd
conclusions. Daniel Bernoulli's 1738 response to the paradox is
presented in Sec.~\ref{Bernoulli's}, followed by a reminder of the
more recent concept of ergodicity in Sec.~\ref{Ergodicity}, which
leads to an alternative resolution in Sec.~\ref{Resolution} with the
key Theorem~\ref{time_theorem}. Section~\ref{Relation} explains the
intriguing relation of this mathematically similar but conceptually
wholly different resolution to Bernoulli's work and resolves
difficulties with unbounded utility functions noted by
\cite{Menger1934}. Section~\ref{Discussion} concludes that the
prominence in economics of Bernoulli's early arguments may have
contributed to poor risk assessments of modern financial products,
with consequences for market stability through the effect of credit
and leverage, as foreseen by writers as early as Adam Smith.

\section{Origins of probability theory}
\label{Origins}
Formal and concrete concepts of likelihood were first developed in the
context of gambling -- notable are the works by \cite{dePacioli1494},
by Cardano in the mid-16$^{\text{th}}$ century \citep{Ore1953}, and
the work by Pascal and Fermat in the summer of 1654.  A prominent
question treated by \cite{dePacioli1494} as well as Pascal and Fermat
(1654) is the following ``problem of the points''\footnote{See
\citep{Devlin2008} for a detailed historical account.}: imagine that a
game of dice has to be abandoned before it can be concluded. For
instance, players may be betting money on the highest score in rolling
a dice three times but have to leave after two throws. In this
situation, how is the ``pot'', the total wager, to be distributed
among the players in a fair manner?

The first observation is that this is a moral question. Mathematics
may aid in answering it, but cannot resolve it without appealing to
external information, as any answer must depend on the concept of
fairness. It could be perceived as fair that the player with the most
points is entitled to the total wager. Another concept of fairness
would be to call the game inconclusive and return to each player his
or her individual wager, or the pot may be split equally between the
participants. Apparently, at least until the $17^{\text{th}}$ century,
there was no universal agreement on the relevant concept of
fairness. \cite{dePacioli1494}, for instance, argued that the fair
solution is to divide the pot in proportion to the points that each
player has accrued when the game is interrupted, see
\citep{Devlin2008}, p.~15.

A century and a half later Pascal was approached by Chevalier de
M\'er\'e to produce a conclusive argument based on mathematics that
would settle the issue. Pascal and Fermat corresponded on the subject
and agreed that the fair solution is to give to each player the
expectation value of his winnings. The expectation value they computed
is an ensemble average, where all possible outcomes of the game are
enumerated, and the product of winnings and probabilities associated
with each outcome for each player are added up. This procedure uses
the then revolutionary idea of parallel universes. Instead of
considering only the state of the universe as it is, or will be, an
infinity of additional equally likely universes is imagined. Any of
these additional universes, for all we know, could be reality (\ie the
world as it will be). The proportion of those universes where some
event occurs is the probability of that event. We will see that in the
$18^{\text{th}}$ century Bernoulli noticed undesired properties of the
ensemble average, and in the $19^{\text{th}}$ century Boltzmann began
to specify conditions for its applicability.

The 1654 investigation, which is generally considered the beginning of
probability theory, was concerned with a specific problem. It did not
attempt to make any predictions, for instance involving repetitions of
the game, but solely gave quantitative guidelines where individuals
had incompatible moral intuitions. Moral considerations were
certainly at the heart of the early debate. Pascal famously used
expectation values to argue in favor of his religious beliefs, and
much of Cardano's work on gambling is concerned with morals. He came
very close to defining a fair game as one where no player has an
expected advantage over others: ``To the extent to which you depart
from that equality, if it is in your opponent's favor, you are a fool,
if in your own, you are unjust'' \citep{Ore1953} p.~189.

Following Pascal's and Fermat's work, however, it did not take long
for others to recognize the potential of their investigation for
making predictions. \cite{Halley1693}, writing in these pages 318
years ago, built on earlier work by \cite{Graunt1662}, and devised a
method for pricing life annuities. The idea of embedding reality in
infinitely many possible alternatives was revolutionary in 1654, it
was essential in the development of statistical mechanics in the
19$^{\text{th}}$ century \citep{Ehrenfest1912,Cohen1996}, and it
continues to be a fruitful means of conceptualizing complex and
stochastic systems \citep{Gell-MannLloyd2004}. Nonetheless the idea
itself is a dubious philosophical construct, justified empirically by
the success that, under appropriate conditions, comes with allowing
the use of mathematical rigor. Historically, it seems that the
philosophical weakness was initially ignored in applications. In
Sec.~\ref{Ergodicity} we will review an alternative conceptualization
of randomness.

\cite{Huygens1657} is credited with making the concept of expectation
values explicit and with first proposing an axiomatic form of
probability theory. This was helpful in developing the field
mathematically, as results could now be proven to be correct. On the
other hand, by introducing an axiomatic system, correctness becomes
restricted to the context of the axioms themselves. A proven result in
probability theory follows from the axioms of probability theory, now
usually those of \cite{Kolmogorov1933}. It is related to reality only
insofar as the relevant real conditions are reflected by the
axioms. \cite{Kolmogorov1933} wrote ``The theory of probability [..]
should be developed from axioms in exactly the same way as Geometry
and Algebra. This means that after we have defined the elements to be
studied and their basic relations, and have stated the axioms by which
these relations are to be governed, all further exposition must be
based exclusively on these axioms, independent of the usual concrete
meaning of these elements and their relations.'' He wrote that it
would be a different ``aim [..] to tie up as closely as possible the
mathematical theory with the empirical development of the theory of
probability.''

To summarize: the first systematic investigation into stochastic
systems was concerned with moral advice. The second established an
axiomatic system.

\section{The lottery}
\label{lottery}
The St. Petersburg paradox was first put forward by Nicolaus Bernoulli
in 1713 \citep{Montmort1713} p.~402. He considered lotteries of the
following type:

A fair coin is tossed.

1) On heads, the lottery pays \$1, and the game ends.
On tails, the coin is tossed again.

2) On heads, the lottery pays \$2, and the game ends.
On tails, the coin is tossed again.

$\cdots$

n) On heads, the lottery pays \$$2^{n-1}$, and the game ends.
On tails, the coin is tossed again.

$\cdots$

In other words, the random number of coin tosses, $n$, follows a
geometric distribution with parameter $1/2$, and the payouts increase
exponentially with $n$. We may call $n$ a ``waiting time'', although
in this study it is assumed that the lottery is performed
instantaneously, \ie a geometric random variable is drawn and no
significant physical time elapses. The expected payout from this game
is
\begin{equation}
\$\sum_{n=1}^{\infty} \left(\frac{1}{2}\right)^n 2^{n-1} =\$\left( \frac{1}{2}+ \frac{1}{2}+...\right),
\elabel{ensemble}
\end{equation}
which is a diverging sum. A rational person, N. Bernoulli argued,
should therefore be willing to pay any price for a ticket in this
lottery. In reality, however, people are rarely willing to pay more
than \$10, which constitutes the paradox. 

Reactions to the paradox include the following:

Even though the expected payout is infinite, there is not an infinite
amount of money or goods in the world to pay up. So the lottery is not
realistic \citep{Cramer1728}. If the payouts are limited to some
realistic value, then the lottery's expected payout is drastically
reduced. For example the 31$^{\text{st}}$ term in the sum
\eref{ensemble} comes from a payout of about \$$10^9$, so limiting
payouts to \$$10^9$ reduces the expected payout from \$$\infty$ to
\$15. Similarly, one could argue that it is only too sensible to ignore
events with a probability of the order of $10^{-9}$ \citep{Menger1934}.

Another argument is that no one would offer such a lottery because it
carries an infinite expected loss for the lottery-seller, which makes
it irrelevant \citep{Samuelson1960}.

\section{Bernoulli's resolution}
\label{Bernoulli's}

The quantity calculated in \eref{ensemble} is usually called an
``expected'' payout. But since it fails to capture the reality of the
situation its conceptual validity must be
questioned. \cite{Bernoulli1738} noted\\ 
``\S1 Ever since mathematicians first began to study the measurement of
risk there has been general agreement on the following proposition:
Expected values are computed by multiplying each possible gain by the
number of possible cases where, in this theory, the consideration of
cases which are all of the same probability is insisted upon.''

Indeed, \cite{Huygens1657} had postulated: ``if any one should put 3
shillings in one hand without telling me which, and 7 in the other,
and give me choice of either of them; I say, it is the same thing as
if he should give me 5 shillings...'' This concept of expectation is
agnostic regarding fluctuations, which is harmless only if the
consequences of the fluctuations, such as associated risks, are
negligible. This is usually the case in small-stakes recreational
gambling as considered in the earliest studies of chance by
\cite{dePacioli1494}, Cardano \citep{Ore1953}, and
\cite{FermatPascal1654}, mentioned in Sec.~\ref{Origins}, but it is
not the case in the St. Petersburg paradox. Noticing that the ability
to bear risk depends not only on the risk but also on the riskbearer's
resources, \cite{Bernoulli1738} wrote under \S3:\\ ``If I am not wrong
then it seems clear that all men cannot use the same rule to evaluate
the gamble. The rule established in \S1 must, therefore, be
discarded.''

Bernoulli, and shortly before him \cite{Cramer1728}, drew
attention to psychological and behavioral issues involved in the
evaluation of the proposed lottery. The desirability or ``utility''
associated with a financial gain, they argued, depends not only on the
gain itself but also on the wealth of the person who is making this
gain. Instead of computing the expectation value of the monetary
winnings, they proposed to compute instead the expectation value of
the gain in utility. To this end the utility function $u(w)$ was
introduced, which specifies the utility of a wealth of \$$w$.

Since an extra dollar is generally worth less to a rich person than to
a poor person, $u(w)$ is assumed to be concave, such that
$\frac{du(w)}{dw}$ is monotonically decreasing. While exceptional
circumstances can render this assumption invalid (Bernoulli cites an
imprisoned rich man who only needs another 2,000 ducats to buy his
freedom), it is well confirmed behaviorally. Otherwise $u(w)$ is only
loosely constrained. \cite{Bernoulli1738} suggested the logarithmic
function $u_B(w)=\ln(w)$, while \cite{Cramer1728} had proposed using
$u_C(w)=\sqrt{w}$ instead. Bernoulli's proposition of the logarithm
was based on the intuition that the increase in wealth should
correspond to an increase in utility that is inversely proportional to
the wealth a person already has, $\frac{du}{dx}=\frac{1}{x}$, whose
solution is the logarithm.
 
\cite{Bernoulli1738} thus ``discarded the rule'' (for calculating
expected gains in wealth) by replacing the object whose expectation
value was to be calculated. Instead of gains in wealth, he decided to
focus on the expectation of gains in some function of wealth.

In Sec.~\ref{Resolution} we will also discard the rule established in
\S1 of \citep{Bernoulli1738}, but not by replacing the object whose
average is to be calculated, \ie not by replacing plain monetary gains
by a function of those gains. Instead we will replace the type of
average, using the time average instead of the ensemble average. This
is necessary because the system under investigation (the dynamics of
monetary wealth) is not ergodic, as will be shown in
Sec.~\ref{Resolution}. In doing so we will critique the implicit
considering of multiple imagined systems, or parallel universes.

But first, applying Bernoulli's reasoning, we compute the expected
change in logarithmic utility, $\ave{\Delta u_B}$, due to playing the
lottery, given the initial wealth $\$w$ and the cost of a ticket in
the lottery $\$c$,
\begin{equation}
 \ave{\Delta u_B}=\sum_{n=1}^\infty \left(\frac{1}{2}\right)^{n}
 \left(\overbrace{\ln(w-c+2^{n-1})}^{\text{Utility after the game}}-
 \underbrace{\ln(w)}_{\text{Utility before the game}}\right).
\elabel{utility_change}
\end{equation}
This sum converges (as long as each individual term is finite), as is
readily shown using the ratio test. Depending on $w$ and $c$, the
quantity can be positive or negative, reflecting expected gain or loss
of utility. Assuming that potential lottery players base their
decisions not on the expected monetary gain but instead on the
expected gain in usefulness, and that that usefulness is appropriately
represented by $u_B$, the paradox is thus resolved.

It is dissatisfying that this resolution of the paradox relies on a
function $u(w)$ that is postulated and, in the framework of Cramer and
Bernoulli cannot be derived from more fundamental
considerations. Disagreements on whether the assumptions (the
characteristics of diminishing marginal utility of wealth) are
realistic are difficult to settle. Anticipating this objection,
\cite{Bernoulli1738} -- Daniel being perhaps less mathematician than
scientist -- appealed to observations: ``Since all our propositions
harmonize perfectly with experience it would be wrong to neglect them
as abstractions resting upon precarious hypotheses.''

The responses to the paradox mentioned at the end of
Sec.~\ref{lottery} are similarly dissatisfying -- they address the
relevance of the problem and argue that it would never really arise,
but they do not resolve it. Since the paradoxical aspect is the
behavior of real people, however, these arguments are valid, and all
means of disposing of the paradox could be similar in character.

While Bernoulli's observations of human risk aversion and even the
functional form he proposed for modeling these are ``correct'' in a
specific sense elaborated in Sec.~\ref{Resolution}, these behavioral
regularities have a physical reason that Bernoulli failed to point
out. In fact, it appears that he was not aware of this physical
reason, which justifies only $u_B(w)=\ln(w)$. \cite{Bernoulli1738} did
not consider the logarithmic form of utility essential and wrote of
Cramer's work, which uses $u_C(w)=\sqrt{w}$: ``Indeed I have found
his theory so similar to mine that it seems miraculous that we
independently reached such close agreement on this sort of subject.''

\section{Ergodicity}
\label{Ergodicity}
The question of ergodicity in stochastic systems is concerned with a
conceptual choice in giving meaning to quantitative probabilities. It
can be argued that it is meaningless to assign a probability to a
single event, and that any decision regarding a single event must
resort to intuition or morals. For mathematical guidance the event has
to be embedded within other similar events. \cite{FermatPascal1654}
chose to embed within parallel universes, but alternatively -- and
often more meaningfully -- we can embed within time. The concept of a
decision regarding a single isolated event, whether probabilistic or
not, seems dubious: how do we interpret the premise of isolation?
Surely the event is part of a history. Does the individual making the
decision die immediately after the event? In general the consequences
of the decision will unfold over time.

The origins of ergodic theory lie in the mechanics of gases
\citep{Uffink2004}. One is interested in large-scale effects of the
molecular dynamics, \ie in the thermodynamic variables. For instance,
the macroscopic pressure of a gas is a rate per area of molecular
momentum transfer to a container wall, averaged over an area that is
large compared to the typical distance between molecules and over a
time that is long compared to the typical interval between
molecular impacts in the area.

Since the number of particles is large and collisions are possible,
however, it is practically not possible to explicitly solve the
microscopic equations of motion. Full information about the state $\x$
(positions and momenta of all molecules) is not available, and
the time average, for instance of momentum transfer to a container
wall, cannot be computed directly. \cite{Boltzmann1871b} and
\cite{Maxwell1879} independently replaced the physically required time
average by the average over an ensemble of appropriately weighted
states $\x$, making use of Huygens' expectation value. The weight of
the different states $\x$ in the ensemble was postulated and
subsequently justified empirically by comparing predictions to
observations.

The key rationale behind this dramatic step is that the systems
considered are in equilibrium: the macroscopic variables of interest
do not change in time, and microscopic fluctuations obey detailed
balance, see \eg \citep{vanKampen1992}. Under these strict conditions,
time has little tangible effect, and we may get away with disregarding
it completely. Nonetheless, both Boltzmann and Maxwell were concerned
that for mathematical convenience they were using the {\it a priori}
irrelevant ensemble average.

Specifically, when \cite{Boltzmann1871b} suggested to treat a gas as a
collection of many systems, namely sub-volumes which can be thought of
as a probabilistic ensemble, he warned that using this ``trick'' means
``to assume that between [...]  the various [...]  systems
\underline{no interaction ever occurs}.'' The requirement of
absolutely no interaction between a collection of systems is
equivalent, in practical terms, to the non-existence of all these
systems from each other's perspectives -- if systems A and B cannot
ever interact in any way, then to system A, for all practical
purposes, system B does not exist, and vice versa. Another way of
putting this is that systems A and B are parallel universes.

Assuming the validity of this procedure is known as the ergodic
hypothesis. It is permissible under strict conditions of stationarity,
see \eg \cite{GrimmetStirzaker2001}, Ch.~9.5.  These conditions were
understood long after the St. Petersburg paradox had been introduced
\citep{Birkhoff1931,Birkhoff1931b,vonNeumann1932A,vonNeumann1932B}.

Much of the literature on ergodic systems is concerned with
deterministic dynamics, but the basic question whether time averages
may be replaced by ensemble averages is equally applicable to
stochastic systems, such as Langevin equations or lotteries. The
essence of ergodicity is the question whether the system when observed
for a sufficiently long time $t$ samples all states in its sample
space in such a way that the relative frequencies $f(\x,t)d\x$ with
which they are observed approach a unique (independent of initial
conditions) probability, $P(\x)d\x$,
\begin{equation}
\lim_{t \to \infty} f(\x, t) = P(\x). 
\elabel{ergodic}
\end{equation}
If this distribution does not exist or is not unique, the time average,
$\bar{A}=\lim_{T\to\infty}\frac{1}{T}\int_0^TA(\x,t) dt$, of an observable
$A$ cannot be computed as an ensemble average in Huygens' sense,
$\ave{A}=\int_\x A(\x,t) P(\x) d\x$. The generic variable $A$ may depend
on time only through its state dependence, or it may have explicit
time dependence. If $P(\x)$ is not unique, then the time average of
$A$ generally depends on initial conditions. If $P(\x)$ does not
exist, there may still be a unique time average. A unique ensemble
average may also still exist -- although we cannot find $P(\x)$ from
\eref{ergodic}, we may be able to determine $\tilde{P}(\x, t)$, the
proportion of systems in an ensemble that are in state $\x$ at time
$t$, and compute the ensemble average as $\ave{A}(t)=\int_\x A(\x,t)
\tilde{P}(\x,t) d\x$. In special cases the time dependencies of
$A(\x,t)$ and $\tilde{P}(\x,t)$ can be such that $\ave{A}(t)$ does not
actually depend on time. However, there is no guarantee in these cases
that the time average and ensemble average will be identical.

Growth factors in the St. Petersburg lottery are such a special
case. In Sec.~\ref{Resolution} it will be shown that while the ({\it a
  priori} irrelevant) ensemble-average winnings from the lottery
diverge, the time-average winnings do not. Mathematically the end
result is identical to the result obtained by Bernoulli (although see
Sec.~\ref{Relation}(\ref{Menger})). Conceptually, however, the
arbitrary utility (arbitrary in the sense that it depends on personal
characteristics), is replaced by an argument based on the physical
reality of the passing of time and the fact that no communication or
transfer of resources is possible between the parallel universes
introduced by Fermat.

\subsection{The economic context}
\label{economic}
To repeat, the quantity in \eref{ensemble} is accurately interpreted
as follows: imagine a world of parallel universes defined such that
every chance event splits our current universe into an ensemble
containing member-universes for every possible outcome of the chance
event. We further require that the proportion of members of the ensemble
corresponding to a particular outcome is the probability of that
outcome. In this case, if we give the same weight to every
member-universe, \eref{ensemble} is the ensemble average over all
possible future states of the universe (\ie states after the game).

Of course, we are not \apriori interested in such an average because
we cannot realize the average payout over all possible states of the
universe. Following the arguments of Boltzmann and Maxwell, this
quantity is meaningful only in two cases.

\begin{enumerate}
\item The physical situation could be close to an ensemble of
  non-interacting systems which eventually share their resources. This
  would be the case if many participants took part in independent
  rounds of the lottery, with an agreement to share their payouts,
  which would be a problem in portfolio construction, and different
  from Bernoulli's setup\footnote{This situation is equivalent to a
    single person buying tickets for many parallel rounds of the
    lottery. In the limit of an infinitely rich person and a finite
    ticket price it can be shown that it is advisable to invest all
    funds in such independent lotteries.}.
\item The ensemble average could reflect the time-average performance
  of a single participant in the lottery. Whereas time averages in
  statistical mechanics are often difficult to compute (hence the
  ergodic hypothesis), the simplicity of the St. Petersburg lottery
  makes it easy to compute them and see how they differ from
  ensemble averages.
\end{enumerate}
Thus neither case applies to the St. Petersburg lottery, and the
ensemble average is irrelevant to the decision whether to buy a
ticket.

In general, to realize an average over the ensemble, ensemble members
must exchange resources, but this is often impossible, so we must be
extremely careful when interpreting ensemble averages of the type of
\eref{ensemble}.

\section{Resolution using non-ergodicity}
\label{Resolution}
The resolution of the St. Petersburg paradox presented in this section
builds on the following alternative conceptualization:
\begin{itemize}
\item
\underline{Rejection of parallel universes:} To the individual who
decides whether to purchase a ticket in the lottery, it is irrelevant
how he may fare in a parallel universe. Huygens' (or Fermat's)
ensemble average is thus not immediately relevant to the problem.
\item
\underline{Acceptance of continuation of time:} The individual
regularly encounters situations similar to the St. Petersburg
lottery. What matters to his financial well-being is whether he makes
decisions under uncertain conditions in such a way as to accumulate
wealth {\it over time}.
\end{itemize}
Similarly, in statistical mechanics Boltzmann and Maxwell were
interested in momentum accumulated {\it over time}. Because they
considered equilibrium systems, where time is largely irrelevant, they
hypothesized that time averages could be replaced by ensemble
averages. However, a person's wealth is not usually in equilibrium,
nor even stationary: on the time scales of interest, it generally
grows or shrinks instead of fluctuating about a long-time average
value. Therefore the ergodic hypothesis does not apply
\citep{Peters2010}. Consequently, there is no reason to believe that
the expected (ensemble-average) gain from the lottery coincides with
the time-average gain. That they are indeed different will be shown in
this section by explicitly calculating both.

The accumulation of wealth over time is well characterized by an
exponential growth rate. To compute this, we consider the factor $r_i$
by which a player's wealth changes in one round of the
lottery\footnote{One ``round'' of the lottery is used here to mean one
  sequence of coin tosses until a tails-event occurs. Throughout, an
  index $i$ refers to such rounds, whereas $n$ indicates waiting times
  -- the number of times a coin is tossed in a given round.},
\begin{equation} 
r_i=\frac{w-c+m_i}{w},
\elabel{ri}
\end{equation} 
where, as in \eref{utility_change}, \$$w$ is the player's wealth
before the round of the lottery, \$$c$ is the cost of a lottery
ticket, and \$$m_i$ is the payout from that round of the lottery. To
convert this factor into an exponential growth rate $g$ (so that
$\exp(gt)$ is the factor by which wealth changes in $t$ rounds of the
lottery), we take the logarithm, $g_i=\ln(r_i)$.\footnote{The
  logarithm is taken to facilitate comparison with Bernoulli's
  analysis. As long as it acts on the average, as opposed to being
  averaged over, it does not change the convergence properties.}

\subsection{Ensemble Average}
\label{Average}
\begin{theorem}
The ensemble-average exponential growth rate in the St. Petersburg lottery is
$\ave{g}=\ln \left(\sum_{n=1}^{\infty}\left(\frac{1}{2}\right)^n\frac{w-c+2^{n-1}}{w}\right).$
\end{theorem}
\begin{proof}
First, we consider the ensemble-average growth factor, and begin by
averaging over a finite sample of $N$ players, playing the lottery
in parallel universes, \ie in general players will
experience different sequences of coin tosses,
\begin{equation}
\ave{r}_N=\frac{1}{N}\sum_{i=1}^N r_i,
\elabel{ensemble_growth_1}
\end{equation}
which defines the finite-sample average $\ave{\cdot}_N$. We change the
summation in \eref{ensemble_growth_1} to run over the geometrically
distributed number of coin tosses in one round, $n$,
\begin{equation}
\ave{r}_N=\frac{1}{N}\sum_{n=1}^{n_N^{\text{max}}} k_n r_n,
\elabel{ensemble_growth_2}
\end{equation}
where $k_n$ is the frequency with which a given $n$, \ie the first
tails-event on the $n^{\text{th}}$ toss, occurs in the sample of $N$
parallel universes, and $n_N^{\text{max}}$ is the highest $n$ observed
in the sample. Letting $N$ grow, $k_n/N$ approaches the probability of
$n$, and we obtain a simple number, the ensemble-average growth factor
$\ave{r}$, rather than a stochastic variable $\ave{r}_N$
\begin{align}
\ave{r}=\lim_{N\to\infty}\ave{r}_N&=\lim_{N\to\infty}\sum_{n=1}^{n_N^{\text{max}}} \frac{k_n}{N} r_n \elabel{ensemble_growth_3}\\
&=\sum_{n=1}^\infty p_n r_n.\nonumber
\end{align}
The logarithm of $\ave{r}$ expresses this as the ensemble-average
exponential growth rate. Using \eref{ri} and writing the probabilities
explicitly, we obtain
\begin{equation}
\ave{g}=\ln \left(\sum_{n=1}^{\infty}\left(\frac{1}{2}\right)^n\frac{w-c+2^{n-1}}{w}\right).
\elabel{ensemble_growth}
\end{equation}
\end{proof}
Since the ensemble-average payout from one round in the lottery
diverges, \eref{ensemble}, so does this corresponding ensemble-average
exponential growth rate \eref{ensemble_growth}. 

\subsection{Time average}
\begin{theorem}
\label{time_theorem}
The time-average exponential growth rate in the St. Petersburg lottery
is $\bar{g}=\sum_{n=1}^\infty \left(\frac{1}{2}\right)^{n}
\ln(w-c+2^{n-1})-\ln(w)$.
\end{theorem}
\begin{proof}
The time average is computed in close analogy to the way the ensemble
average was computed. After a finite number $T$ of rounds of the game
a player's wealth reaches\footnote{A possible objection to this
  statement is discussed below, starting on p.~\pageref{product_objection}.}
\begin{equation}
w(T)=w \prod_{i=1}^T r_i.
\end{equation}
The $T^{\text{th}}$ root of the total fractional change,
\begin{equation}
\bar{r}_T=\left(\prod_{i=1}^T r_i\right)^{1/T},
\end{equation}
which defines the finite-time average $\bar{r}_T$, is the factor by
which wealth has grown on average in one round of the lottery over the
time span $T$. We change the product to run over $n$,
\begin{equation}
\bar{r}_T=\left(\prod_{n=1}^{n^{\text{max}}_T} r_n^{k_n}\right)^{1/T},
\elabel{time_factor_1}
\end{equation}
where $k_n$ is the frequency with which a given $n$ occurred in the
sequence of $T$ rounds, and $n^{\text{max}}_T$ is the highest $n$
observed in the sequence. Letting $T$ grow, $k_n/T$ approaches the
probability of $n$, and we obtain a simple number, the
time-average growth factor $\bar{r}$, rather than a stochastic
variable $\bar{r}_T$
\begin{align}
\bar{r}&=\lim_{T\to \infty}\bar{r}_T=\lim_{T\to \infty}\prod_{n=1}^{n^{\text{max}}_T} r_n^{k_n/T}\elabel{time_factor}\\
&=\prod_{n=1}^\infty r_n^{p_n}.\nonumber
\end{align}

The logarithm of $\bar{r}$ expresses this as the time-average
exponential growth rate. Using \eref{ri} and writing the probabilities
explicitly, we obtain
\begin{align}
\bar{g}&=\ln\left(\prod_{n=1}^\infty r_n^{p_n}\right) \elabel{time_growth}\\
&=\sum_{n=1}^\infty p_n \ln r_n \nonumber\\
&=\sum_{n=1}^\infty \left(\frac{1}{2}\right)^{n} \left(\ln(w-c+2^{n-1})-\ln(w)\right). \nonumber
\end{align}
\end{proof}
The final line of \eref{time_growth} is identical to the right-hand
side of \eref{utility_change}.

Again the quantity can be positive or negative, but instead of the
ensemble average of the change in utility we have calculated the
time-average exponential growth rate of a player's wealth without any
assumptions about risk preferences and personal characteristics. If
the player can expect his wealth to grow over time, and he has no
other constraints, he should play the game; if he expects to lose
money over time, he should not play. The loci of the transition
between growth and decay, where $\bar{g}=0$ define a line in the $c$
{\it vs.} $w$ plane, which is shown in
\fref{g_bar}. \Eref{time_growth} depends on the player's wealth --
keeping $c>1$ fixed, $\bar{g}$ initially increases with $w$, see inset
of \fref{g_bar} for the example of $c=2$. This is because the wealthy
player keeps most of his money safe, and a loss does not seriously
affect his future ability to invest. For the very wealthy player,
neither a win nor a loss is significant, and the time-average
exponential growth rate asymptotes to zero as $w\to\infty$. At the
other extreme, a player whose wealth $w\leq c-1$ risks bankruptcy,
which in a sense means the end to his economic life, and $\lim_{w\to
  (c-1)^+}\bar{g}=-\infty$.  

\begin{figure}
\includegraphics[width=.9\textwidth]{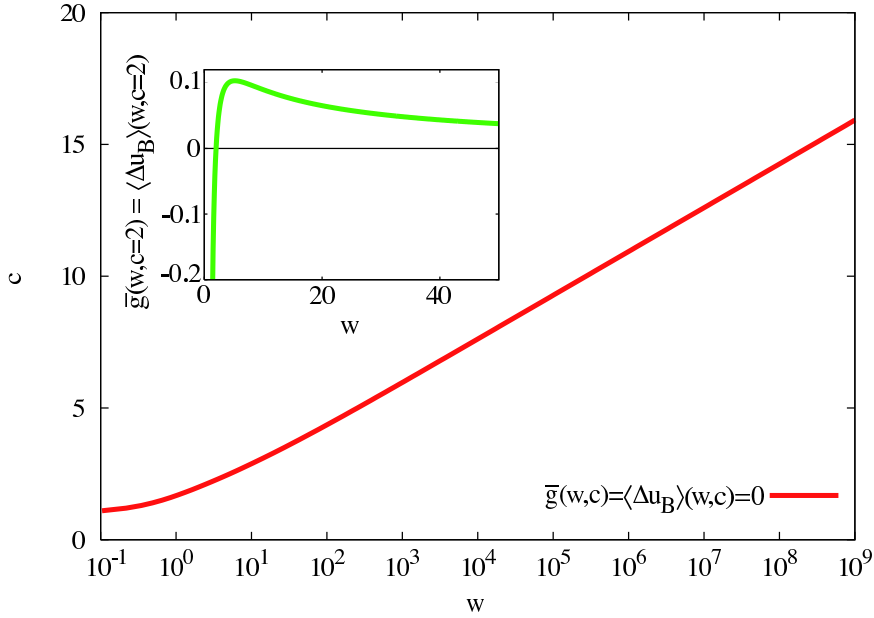}
\caption{\Eref{time_growth} (or \eref{utility_change}) defines a
  relationship between $w$ and $c$, where $\bar{g}(w,c)=0$, \ie the
  player breaks even over time (or his ensemble-average logarithmic
  utility change is zero) if he pays $\$c$ given his wealth
  $\$w$. \underline{Inset:} Time-average exponential growth rate (or
  ensemble-average logarithmic utility change), $\bar{g}(w,
  c=2)=\ave{\Delta u_B}(w, c=2)$, for the St. Petersburg lottery as a
  function of wealth, $\$w$, with a ticket price of $\$c=\$2$. If the
  player risks bankruptcy by purchasing a ticket, $\bar{g}(w,c) \to
  -\infty$. To the infinitely wealthy player a gain or loss is
  irrelevant, and $\bar{g}(w,c) \to 0$.}  \flabel{g_bar}
\end{figure}

So \eref{time_growth} can also be considered a criterion for how much
risk a person should take. The cost of the ticket, $c$, is the
exposure to the lottery. For fixed (positive) $w$ ``buying'' a ticket
is always advisable if $c=0$, see
Sec.~\ref{Relation}(\ref{Menger}). As $c$ increases, $\bar{g}$ will
eventually become negative as the exposure, or risk, becomes too
large.  The time resolution discourages entering into any gamble where
bankruptcy, \ie zero or negative wealth after the game, occurs with
non-zero probability. In these cases individual terms in the sum in
\eref{time_growth} are undefined.

\label{product_objection}
\Eref{time_growth}, may seem an unnatural criterion for the following
reason: the lottery is played only once with wealth $w$. By the next
round, wealth has changed by $m_i-c$, and the situation has to be
re-evaluated. So in reality, although we may buy tickets at the same
price repeatedly, a combination of different $\bar{g}$ (resulting from
different ``initial'' wealths $w$) will be realized.  However, over a
sufficiently long time, we must assume that we will face equivalent
decisions again, and thus play equivalent lotteries again. Let's
consider the result of playing many rounds in different lotteries,
$\prod_i^T r_i$. Because of commutativity we can rearrange the factors
in the product so that the first $T'$ factors correspond to the rounds
in which we face equivalent lotteries (for instance we have the same
wealth, and the tickets are offered at the same price), and the
remaining $T-T'$ factors refer to different situations,
\begin{align}
\prod_{i=1}^{T}r_i &=\prod_{j=1}^{T'} r_j \prod_{m=T'+1}^Tr_m.
\end{align}
Whatever $\prod_{m=T'+1}^Tr_m$ may be, the steps in \eref{time_factor}
apply to the first product, and the sign of the quantity in
\eref{time_growth}, which determines whether the first product is
greater or smaller than one, is a good criterion for deciding whether
to buy a ticket.

It is instructive to calculate the time-average exponential growth
rate in another way.  Line 2 of \eref{time_growth} looks like an
ensemble-average exponential growth rate, obtained by computing
exponential growth rates for individual systems and then averaging
those over the ensemble,
\begin{equation}
\lim_{N\to\infty}\frac{1}{N}\sum_{i=1}^N \ln(r_i).
\elabel{ensemble_wrong}
\end{equation}
The reason why this quantity is not the ensemble average but the time
average is subtle. There is no limit $T \to \infty$, so how can this
be a time average?

Averages extract deterministic parameters from stochastic
processes. The ensemble average does this by considering an infinity
of parallel universes, and the time average does it by considering
infinitely many time steps. But to find the time average it is not
necessary for time itself to approach infinity. Instead, the unit of
time can be rescaled. It is only necessary that all possible scenarios
occur exactly with the appropriate frequencies during the sampling
time. As long as the effect of time -- of events happening
sequentially -- is accounted for, this will lead to the time average.

We have used one round of the lottery as one time unit. Thus the
return from one time unit will be one of the possible returns
$r_n$. If we observe only one time unit, the best estimate for the
time-average return would be the return that happened to be realized
in that time step. An approximate estimate for the time-average
exponential growth rate is thus $\gest_1\approx r_1-1$\footnote{The
  approximation $\ln(r_1)\approx r_1-1$ is permissible here because we
  will consider infinitesimal time steps such that
  $\exp(\bar{g}dt)=1+\bar{g}dt$ will be exact.}.

To improve this estimate, we pick $q$ returns $r_j$ at random and in
agreement with the probabilities $p_n$, and let each return act for
$1/q$ time units\footnote{The subscript $j$ is used here instead of
  $i$ because it refers not to one round but to sub-intervals of one
  round.}. The total time that passes during the experiment is kept
fixed but we separate it into $q$ sub-intervals of time. The result
will be
\begin{equation}
\gest_q=\sum_{j=1}^q(r_j^{1/q}-1),
\end{equation}
The proportion of sub-intervals during which return $r_n$ is realized
will approach $p_n$ as $q\to\infty$. In this limit we can therefore
replace the sum over time steps by a sum over $n$ as follows
\begin{align}
\gest_\infty&=\lim_{q\to\infty}\sum_{j=1}^q(r_j^{1/q}-1)\\
&=\sum_{n=1}^{\infty} \lim_{q \to\infty} k_n (r_n^{1/q}-1)\\
&=\sum_{n=1}^{\infty} p_n  \lim_{q \to\infty}q (r_n^{1/q}-1),
\end{align}
where $k_n$ once again is the frequency with which a given $n$ occurs,
now in the sample of $q$ sub-intervals. Using the definition of the
logarithm, $\ln(r_n)\equiv\lim_{q \to\infty} q (r_n^{1/q}-1)$ yields
\begin{equation}
\lim_{q\to\infty}\sum_{j=1}^q(r_j^{1/q}-1)=\sum_{n=1}^{\infty} p_n \ln r_n,
\end{equation}
meaning that the time-average exponential growth rate, derived by
splitting a time unit into infinitely many sub-intervals and playing
through all possible scenarios in these sub-intervals can be written
as the expectation value of the logarithm of returns. A limit which is
equivalent to the apparently missing limit $T\to\infty$ is implied by
the logarithm and evaluated before the explicit limit in
\eref{ensemble_wrong}. Thus \eref{ensemble_wrong} is an ensemble
average of a time average, which is nothing but a time average.

\section{Relation to Bernoulli's resolution}
\label{Relation}
\Eref{time_growth} is mathematically equivalent to Bernoulli's use of
logarithmic utility. Bernoulli argued behaviorally that instead of the
expectation value of monetary gain, the expectation value of the gain
in a loosely constrained function (the utility) of wealth should be
considered. One of the allowed functions is the logarithm, which has
the special property of encoding the multiplicative nature common to
gambling and investing in a linear additive object, the expectation
value
\begin{equation}
\sum_{n=1}^{\infty} p_n\ln r_n=\ln\left(\lim_{T\to \infty}\left(\prod_{i=1}^T r_i \right)^{1/T}\right).
\elabel{relation}
\end{equation}
 Inadvertently, by postulating logarithmic utility (left-hand side of
 \eref{relation}), Bernoulli replaced the ensemble-average winnings,
 with the time-average exponential growth rate in a multiplicative
 non-ergodic stochastic process (right-hand side of \eref{relation}).

Bernoulli did not make the time argument, as is evident from his
acceptance of Cramer's square-root utility, which does not have the
property of \eref{relation}: $\sum_i^{\infty} p(r_i) \sqrt{r_i}$
cannot be written as a similar product. This is problematic because
the arbitrariness of utility can be abused to justify reckless
behavior, and it ignores the fundamental physical limits, given by
time irreversibility, to what can be considered reasonable. But
because Bernoulli's early work postulated the logarithm, many
consequences of \eref{time_growth} have already been discussed in the
literature under the heading of ``logarithmic utility''.

A different heading for these results is ``Kelly criterion''
\citep{Kelly1956,CoverThomas1991,Thorp2006}. In contrast to
ensemble-average exponential growth rates, which often diverge (for
instance with leverage) time-average exponential growth rates can be
optimized \citep{Peters2009,Peters2010}. \cite{Kelly1956} used this
fact to optimize wager sizes in a hypothetical horse race using
private information. While he refrained from using utilities because
he deemed them ``too general to shed any light on the specific
problems'' he considered, he did not point out the fundamental
difference in perspective his treatment implies: in essence, arbitrary
utility functions are replaced by the physical truth that time cannot
be reversed. My aim here is to emphasize this difference in
perspective. It is crucial that logarithmic utility from this point of
view is not a utility at all. Rather, the logarithm accounts for the
multiplicative nature of the process: the ensemble average of the
logarithm of growth factors equals the logarithm of the time average
of growth factors.

Comparing \eref{time_growth} and \eref{utility_change}, it is tempting
to say that the time average justifies logarithmic utility. I advise
against this interpretation because it conflates physical concepts of
time with behavioral concepts of usefulness. Any utility function
other than the logarithm leads to recommendations that do not agree
with the time perspective.

\subsection{Menger's objection to unbounded utility}
\label{Menger}
\cite{Bernoulli1738} did not actually write down
\eref{utility_change}, although it is often assumed that that was his
intention. Instead of using the criterion \eref{utility_change} he
argued in two steps ``how large a stake an individual should be
willing to venture'', pp.~26--27.
\begin{enumerate}
\item
The expected gain in utility is calculated without explicitly taking
the ticket price into account
\begin{equation}
\elabel{mistake_1}
\sum_{n=1}^\infty
\left(\frac{1}{2}\right)^n
\ln\left(w+2^{n-1}\right) - \ln(w).
\end{equation} 
\item
This is followed by the statement that ``the stake more than which
persons [...]  should not venture'' is that ticket price, $\$c$, which
satisfies
\begin{equation}
\elabel{mistake_2} \overbrace{\sum_{n=1}^\infty
  \left(\frac{1}{2}\right)^n
  \left(\ln(w+2^{n-1})-\ln(w)\right)}^{\text{expected utility gain with
    $c=0$}} -\underbrace{\left[\ln(w)-\ln(w-c)\right]}_{\text{utility loss
    at purchase}}=0.
\end{equation}
\end{enumerate}
This is not the condition that defines the line in the main panel of
\fref{g_bar} as it does not imply that the expected gain in utility,
\eref{utility_change}, is zero. In this sense \eref{utility_change} is
an inaccurate, although generally accepted and sensible,
representation of Bernoulli's work. The difference between
\eref{utility_change} and \eref{mistake_2} and the conflicting
interpretations of these equations are consequences of the
aforementioned arbitrariness of the utility framework.

\cite{Menger1934} claimed that using logarithmic utility as Bernoulli
did, \eref{mistake_1} and \eref{mistake_2}, does not resolve modified
versions of the paradox where payouts $\$f(n)$ as a function of
waiting time $n$ increase faster than according to Bernoulli's
original $f_B(n)\equiv 2^{n-1}$. Specifically, Menger considered
$f_M(n) \equiv w \exp(2^{n})-w$. Note that this function is curiously
defined in terms of the initial wealth. His first step closely mirrors
Bernoulli's first step, but then he jumps to an untenable conclusion:
\\
\begin{enumerate}
\item
\cite{Menger1934} pointed out that replacing $2^{n-1}$ in \eref{mistake_1} by
$f_M(n)$, the expected gain in logarithmic utility at zero ticket
price diverges.
\item
He argued that the paradox remains unresolved because ``it
is clear that also in the modified St. Petersburg game no normal
person will risk a large amount or even his fortune as a wager'',
p.~468, my translation, and generalized to conclude that this formally prohibits
the use of any unbounded utility function.
\end{enumerate}
The meaning of the second statement is unclear. A player who pays
``his fortune'' for a ticket in Menger's lottery and then experiences
a heads-event on the first coin toss, \ie the worst possible outcome,
will still gain since the worst-case payout,
$\$f_M(1)=\$w\exp(2)-\$w$, is more than the initial wealth, $\$w$. For
a person to risk losing anything at all, the ticket price has to be
$\$c\geq \$ f_M(1)$, far greater than the person's wealth. For a
person to risk losing his entire wealth, the ticket price has to be
greater still, $\$c\geq \$ f_M(1)+\$w=\$w\exp(2)$. But at such prices
\eref{mistake_2} is undefined.

Perhaps Menger meant that Bernoulli's condition \eref{mistake_2}
cannot be satisfied, and the price one should be willing to pay is
infinite. In that case the second part of Menger's argument implicitly
assumes that the positive divergence in the first part cannot be
offset by anything else. But as the ticket price approaches the
player's wealth, $c\to w$, the utility loss at purchase in
\eref{mistake_2} represents another divergence, implying that this
argument is inconclusive. It will now be shown to be invalid.

Menger's objection to Bernoulli could hold in the following sense: the
left-hand side of \eref{mistake_2}, using Menger's payout function,
diverges positively if $c<w$ and is undefined otherwise. But it is
never zero (or negative) -- when does it warn against the gamble? To
understand what the undefined regime signifies, one has to study the
process of divergence and compare the two infinities as they unfold.

For any finite $n_{\text{max}}$, a finite value $c<w$ does exist which
renders the corresponding partial sum zero,
\begin{equation}
\elabel{mistake_3} 
\forall n_{\text{max}}<\infty \hspace{.2cm} \exists \hspace{.2cm} c<w :\sum_{n=1}^{n_{\text{max}}}
\left(\frac{1}{2}\right)^n 2^n -\left[\ln(w)-\ln(w-c)\right]=0.
\end{equation}
To ensure positivity up to exactly $c=w$, where the expression becomes
undefined, events of zero probability have to be taken into
account. For any non-zero lower bound on probability Bernoulli's
condition can be satisfied. In this sense, in the literal original
Bernoulli-setup, values of $c\geq w$, where \eref{mistake_3} is
undefined, correspond to the recommendation not to buy a ticket, and
the paradox is resolved.

Menger's conclusion is incorrect. Bernoulli's logarithmic utility
recommends to purchase tickets as long as they cost less than the
player's wealth, implying a significant minimum gain -- a course many
a ``normal person'' may wish to pursue. The criterion could be
criticized for the opposite reason: it rejects prices that guarantee a
win, even in the worst case.

The time resolution produces the criterion in
Theorem~\ref{time_theorem}, which is equivalent to
\eref{utility_change} and not to Bernoulli's literal original
criterion \eref{mistake_2}.
Consequently, it yields a different recommendation, which may at first
appear surprising but turns out also to correspond to reasonable
behavior given the assumptions on which it is based: it recommends to
purchase a ticket at any price that cannot lead to bankruptcy. The
player could be left with an arbitrarily small positive wealth after
one round. The recommendation may be followed by a ``normal person''
because of the assumption that equivalent lotteries can be played in
sequence as often as desired. Under these conditions, irrespective of
how close a player gets to bankruptcy, losses will be recovered over
time. Of course, if these conditions are violated, the time resolution
does not apply. This last statement is another warning against the
na\"{i}ve use of mathematics, whose truth is always restricted to the
context of axioms or assumptions. Applicability reflects the degree to
which assumptions are representative of real conditions in a given
situation. While ensemble averages are meaningless in the absence of
samples (here parallel rounds), time averages are meaningless in the
absence of time (here sequences of equivalent rounds).

\section{Discussion}
\label{Discussion}

Excessive risk is to be avoided primarily because we cannot go back in
time. Behavioral aspects and personal circumstances are relevant on a
different level -- they can change and do not immediately follow from
the laws of physics.

The perspective described here has consequences far beyond the
St. Petersburg paradox, including investment decisions
\citep{Peters2009,Peters2010} as well as macro-economic processes. For
example, it is sensible for a nation striving for growth to encourage
risks that lead to occasional bankruptcies of companies and
individuals. How large a risk is macroeconomically sensible? What are
the moral implications? Does gross domestic product -- a linear sum,
similar to a sample average -- measure what one should be interested
in? While the St. Petersburg lottery is an extreme case,
\eref{time_growth} and \fref{g_bar} carry a more general message: if
net losses are possible, the negative time-average exponential growth
rate for small wealth, $w$, turns positive as $w$ increases, implying
higher exponential growth rates for larger entities. In a collection
of such entities inequality has a tendency to increase, letting large
entities dominate and monopolies arise. This can cause markets to
cease functioning, as competition is compromised or corporations
become ``too big to fail''. There is anecdotal evidence that assets in
stock markets have become more correlated in recent decades, and
effective diversification (which mimics ensembles) harder to
achieve. This would make the time perspective even more important, and the
consequences of ignoring it more dramatic.

Utility functions are externally provided to represent risk
preferences but are unable by construction to recommend appropriate
levels of risk. The framework is self-referential in that it can only
translate a given utility function into actions that are optimal with
respect to that same utility function. This can have unwanted
consequences. For example, leverage or credit represents a level of
risk which needs to be optimized, but current incentive structures in
the financial industry can encourage exceeding the optimal risk. Adam
\cite{Smith1776}, cited in \citep{Foley2006}, warned that excessive
lending -- in his case based on bills of exchange for goods in transit
-- can lead to a collapse of the credit system, followed by
bankruptcies and unemployment; insufficient lending, on the other
hand, can lead to economic stagnation, two stages which often follow
one another in a boom-bust cycle. To avoid both, good criteria for
appropriate levels of risk are needed, which the utility framework
cannot deliver. The time arguments presented here provide an objective
null-hypothesis concept of optimality and can be used to optimize
leverage under a given set of conditions \citep{Peters2010}. In the
present case, optimality based on such considerations is a good
approximation to practically optimal behavior. This is evident from
Bernoulli's work, whose justification of the almost identical result
was practical validity.

The proposed conceptual re-orientation may help reconcile economic
theory with empirical regularities, such as risk aversion, known from
behavioral economics.

It is easy to construct examples where less risk should be taken than
recommended by the criterion in \eref{time_growth}. For example, some
fraction of \$$w$ may already be earmarked for other vital use. It is
very difficult, however, to think of an example where greater risk is
beneficial. For this reason, the time-perspective is a useful tool for
finding upper limits on risk to be implemented in regulation, for
instance as margin requirements, minimum capital requirements, or
maximum loan-to-value ratios. Of course, such regulation must take
into account further arguments, but its least complex form can be
based on the time perspective.

The epistemological situation is typical of the process of concept
formation \citep{Lakatos1976}. As the conceptual context changed, in
this case from moral to predictive, the original definition of the
term ``expectation'' began to lead to paradoxical conclusions. From
today's conceptually later perspective it appears that N. Bernoulli
made a hidden assumption, namely the assumption, explicitly stated by
\cite{Huygens1657}, that ``it is the same thing'' to receive 5
shillings as it is to have an equal chance of receiving either 3 or
7 shillings. \cite{Lakatos1976} points out that it can be hard to
imagine in retrospect that an eminent mathematician made a hidden
assumption, which is often perceived as an error. He writes on p.~46,
``while they [the hidden assumptions] were in your {\it subconscious},
they were listed as {\it trivially true} -- the ...[paradox] however
made them summersault into your conscious list as {\it trivially
  false}.''  Similarly, it seems trivially true at first that the
expected gain from playing the lottery should be the criterion for
participating -- taking time into account makes this assumption
summersault into being trivially false.

Thus the St. Petersburg paradox relies for its existence on the
assumption that the expected gain (or growth factor or exponential
growth rate) is the relevant quantity for an individual deciding
whether to take part in the lottery.  This assumption can be shown to
be implausible by carefully analyzing the physical meaning of the
ensemble average. A quantity that is more directly relevant to the
financial well-being of an individual is the growth of an investment
over time. Utility, which can obscure risks, is not necessary to
evaluate the situation and resolve the paradox. It is the actual
wealth, in \$, of a player, not the utility, that grows with
$\bar{g}$, \eref{time_growth}. It is manifestly not true that the
commonly used ensemble-average performance of the lottery equals the
time-average performance. In this sense the system is not ergodic, and
statements based on anything other than measures of the actual
time-average performance must be interpreted carefully. 

\section*{Acknowledgments}
For discussions and thoughtful comments and suggestions I would like
to thank A. Adamou, M. Gell-Mann, R. Hersh, D. Holm, O. Jenkinson,
W. Klein, J. Lamb, J. Lawrence, C. McCarthy and J.-P. Onstwedder. This
work was supported by ZONlab Ltd.

\end{document}